\documentclass[11pt]{amsproc}
 \usepackage[margin=1in]{geometry}
\usepackage{setspace,fullpage}
\usepackage{graphicx}
\usepackage[nice]{nicefrac}
\usepackage{amssymb}
\usepackage{amsmath,amsthm,amscd,amssymb, mathrsfs}

\usepackage{latexsym}
\usepackage[colorlinks,citecolor=red,pagebackref,hypertexnames=false]{hyperref}

\numberwithin{equation}{section}

\theoremstyle{plain}
\newtheorem{theorem}{Theorem}[section]
\newtheorem{lemma}[theorem]{Lemma}

\theoremstyle{definition}

\theoremstyle{remark}

\newtheorem{case[theorem]}{Case}

\title{On the quotient set of the distance set}
\author{A. Iosevich, D. Koh and H. Parshall}



\begin{document}

\begin{abstract} Let ${\Bbb F}_q$ be a finite field of order $q.$ We prove that if $d\ge 2$ is even and $E \subset {\Bbb F}_q^d$ with $|E| \ge 9q^{\frac{d}{2}}$ then 
$$ {\Bbb F}_q=\frac{\Delta(E)}{\Delta(E)}=\left\{ \frac{a}{b}: a \in \Delta(E), b \in \Delta(E) \backslash \{0\} \right\},$$ where 
$$ \Delta(E)=\{||x-y||: x,y \in E\}, \ ||x||=x_1^2+x_2^2+\cdots+x_d^2.$$ 
If the dimension $d$ is odd and $E\subset \mathbb F_q^d$ with $|E|\ge 6q^{\frac{d}{2}},$ then 
$$ \{0\}\cup\mathbb F_q^+ \subset \frac{\Delta(E)}{\Delta(E)},$$
where $\mathbb F_q^+$ denotes the set of nonzero quadratic residues in $\mathbb F_q.$ Both results are, in general, best possible, including the conclusion about the nonzero quadratic residues in odd dimensions. 
\end{abstract} 
\maketitle

\section{Introduction}

The Erd\H os-Falconer distance problem in vector spaces over finite fields asks for the smallest possible size of 
$$\Delta(E) = \{||x-y||: x,y \in E\}, \ ||x||=x_1^2+\dots+x_d^2,$$ given $E \subset {\Bbb F}_q^d$, $d \ge 2$.  This problem was introduced by Bourgain, Katz and Tao in \cite{BKT04}. Here ${\Bbb F}_q$ denotes the finite field with $q$ elements and ${\Bbb F}_q^d$ is the $d$-dimensional vector space over this field. 

In \cite{IR07}, the first listed author and Misha Rudnev proved that if $E \subset {\Bbb F}_q^d$, $d \ge 2$, with $|E|>2q^{\frac{d+1}{2}}$, then $\Delta(E)={\Bbb F}_q$.  Hart, the first two listed authors, and Rudnev~\cite{HIKR11} showed that, in a sense, this result is best possible when $d$ is odd.  More precisely, for any $c \in (0,1)$ and any $q$ sufficiently large with respect to $c$, they construct subsets $E \subseteq \mathbf{F}_q^d$ with $|E| > \frac{c}{2}q^\frac{d + 1}{2}$ but $|\Delta(E)| < cq$. This construction does not appear to generalize to the even dimensional case. In \cite{CEHIK10}, Chapman, the first listed author, Erdo\u{g}an, Hart and the second listed author proved that if $q$ is prime, $q \equiv 3 \mod 4$ and 
if $E\subset \mathbb F_q^2$ with $|E| \ge Cq^{\frac{4}{3}}$ for a sufficiently large constant $C>0$, then 
$$|\Delta(E)|>\frac{q}{2}.$$  This result was extended to two dimensional vector spaces over arbitrary finite fields in \cite{BHIPR17}. In even dimensions $d \geq 2$, it is reasonable to conjecture that if $|E| \geq Cq^{\frac{d}{2}}$ with a sufficiently large $C$, then $|\Delta(E)| > \frac{1}{2}q$, but this conjecture currently remains open. The exponent $\frac{d}{2}$ cannot be improved. To see this, let $q=p^2$, $p$ prime, and let $E={\Bbb F}_p^d \subset {\Bbb F}_q^d$. Then $|E|=q^{\frac{d}{2}}$, yet $\Delta(E)={\Bbb F}_p$. When $q$ is a prime and $d \ge 4$, the sharpness of $\frac{d}{2}$ can be demonstrated using Lagrangian subspaces (\cite{HIKR11}). In two dimensions, the sharpness of $\frac{d}{2}=1$ is easily demonstrated by taking a suitable subset of a straight line. 

The purpose of this paper is to show that under the assumption $|E| \ge Cq^{\frac{d}{2}}$, taking the quotient of the elements of $\Delta(E)$ recovers all of ${\Bbb F}_q$ for $d$ even, and at least all the square elements of ${\Bbb F}_q$ when $d$ is odd. More precisely, for $E \subseteq \mathbf{F}_q^d$ we define
\[
	\frac{\Delta(E)}{\Delta(E)} := \left\{ \frac{a}{b}: a \in \Delta(E), \ b \in \Delta (E) \backslash \{0\} \right\}.
\]

\vskip.125in 

Our main results are the following. 

\begin{theorem}\label{main} 
Let $E \subset {\Bbb F}_q^d$ $d$ even. Then if $|E| \ge 9q^{\frac{d}{2}}$ , we have 
$$ {\Bbb F}_q=\frac{\Delta(E)}{\Delta(E)}.$$ 
\end{theorem}

\begin{theorem} \label{mainOdd}  Let $d\ge 3$ be an odd integer and $E\subset \mathbb F_q^d.$ Then if  $|E|\ge 6q^{\frac{d}{2}}$,  we have
$$ \{0\}\cup\mathbb F_q^+ \subset \frac{\Delta(E)}{\Delta(E)}.$$
\end{theorem}

\vskip.125in 

\subsection{Sharpness of results} The results are in general sharp up to constants. To see this, we once again take $q=p^2$ and $E={\Bbb F}_p^2$. Then $|E|=q^{\frac{d}{2}}$, yet 
$$ \left\{ \frac{a}{b}: a \in \Delta(E), \ b \in \Delta(E) \backslash \{0\} \right\}={\Bbb F}_p.$$ 

\vskip.125in 

The statement about the squares in Theorem \ref{mainOdd} is also sharp. The example in \cite{HIKR11} (page 15) that illustrates the sharpness of the exponent $\frac{d+1}{2}$ yields a set of size $cq^{\frac{d+1}{2}}$, with $c$ sufficiently small, such that $\Delta(E) \subset \{{(a-a')}^2: a,a' \in A\}$, where $A$ is a suitable arithmetic progression in ${\Bbb F}_q$. In particular, $\Delta(E)$ is a subset of the squares, so the ratios of the elements of $\Delta(E)$ are also squares. 

\vskip.25in 

\section{Proof of Theorem \ref{main}}

\vskip.125in 

For $t\in \mathbb F_q$, let 
$$ \nu(t)=\sum_{x,y\in \mathbb F_q^d} E(x)E(y)S_t(x-y),$$ where 
$$ S_t=\{x \in {\Bbb F}_q^d: ||x||=t\}.$$ 
It is clear that $\displaystyle 0\in \frac{\Delta(E)}{\Delta(E)}$ unless $\Delta(E)=\{0\}.$ Thus it suffices to prove that for each $r\ne 0$ there exists $t\in \Delta(E)\setminus\{0\}$ such that $tr\in \Delta(E).$
Since $t\in \Delta(E)$ if and only if $\nu(t)>0$, we must show that for any $r \in {\Bbb F}_q^*$, 
\begin{equation} \label{mama} \nu^2(0)<\sum_{t\in \mathbb F_q} \nu(t) \nu(rt). \end{equation} 

We shall need the following standard Fourier analytic preliminaries. Given $f: {\Bbb F}_q^d \to {\mathbb C}$, define the Fourier transform $\widehat{f}$ by the formula 
$$ \widehat{f}(m)=q^{-d} \sum_{x \in {\Bbb F}_q^d} \chi(-x \cdot m) f(x),$$ where $\chi$ is a non-trivial principal character on ${\Bbb F}_q$. We shall use the following calculation repeatedly. 
\begin{lemma} \label{fourier} With the notation above, 
$$ \text{(FOURIER INVERSION)} \ f(x)=\sum_{m \in {\Bbb F}_q^d} \chi(x \cdot m) \widehat{f}(m) $$ and 
$$ \text{(PLANCHEREL)} \ \sum_{m \in {\Bbb F}_q^d} {|\widehat{f}(m)|}^2=q^{-d} \sum_{x \in {\Bbb F}_q^d} {|f(x)|}^2.$$ 
\end{lemma} 

\vskip.125in 

By Fourier inversion, 
\begin{align*} \nu(t)&= q^{2d}\sum_{m\in \mathbb F_q^d} \widehat{S_t}(m)|\widehat{E}(m)|^2\\
&=q^{-d}{|E|}^2 |S_t|+q^{2d}\sum_{m\ne \vec{0}} \widehat{S_t}(m)|\widehat{E}(m)|^2.\end{align*} 

It follows that for $r\in \mathbb F_q^*$,
\begin{equation}\label{formulasumproduct} \sum_{t\in \mathbb F_q} \nu(t)\nu(rt)\end{equation}
$$ =\sum_{t\in \mathbb F_q} \left(q^{-d}{|E|}^2 |S_t|+q^{2d}\sum_{m\ne \vec{0}} \widehat{S}_t(m)|\widehat{E}(m)|^2  \right) \left(q^{-d}{|E|}^2 |S_{rt}|+q^{2d}\sum_{m'\ne \vec{0}} \widehat{S}_{rt}(m')|\widehat{E}(m')|^2  \right)  $$
$$ = q^{-2d}|E|^4 \sum_{t\in \mathbb F_q}|S_t||S_{rt}|  + q^d|E|^2 \sum_{m'\ne \vec{0}} |\widehat{E}(m')|^2 \sum_{t\in \mathbb F_q} |S_t| \widehat{S}_{rt}(m') $$
$$ +q^d|E|^2 \sum_{m\ne \vec{0}} |\widehat{E}(m)|^2 \sum_{t\in \mathbb F_q} |S_{rt}| \widehat{S}_{t}(m)
+ q^{4d} \sum_{m,m'\ne \vec{0}} |\widehat{E}(m)|^2|\widehat{E}(m')|^2 \sum_{t\in \mathbb F_q} \widehat{S_{t}}(m)\widehat{S}_{rt}(m')$$
$$:= I + II + III +IV.$$


We shall invoke the explicit value of $|S_t|$ which can be deduced by Theorem 6.26 in \cite{LN97}.
\begin{lemma}\label{sizeSt} Let $S_t \subset \mathbb F_q^d$ denote the sphere with radius $t\in \mathbb F_q.$ Then if $d\ge 2$ is even, 
$$ |S_t|=q^{d-1} + \lambda(t) q^{\frac{d-2}{2}} \eta\left((-1)^{d/2}\right),$$
where $\eta$ is the quadratic character of $\mathbb F_q^*,$  $\lambda(t)=-1 $ for $t\in \mathbb F_q^*,$ and $\lambda(0)=q-1.$\\
\end{lemma}

 We also use the following result which was given as Lemma 4 in \cite{IK10}.
\begin{lemma}\label{Lem1}
Let $S_j$ be a sphere in ${\Bbb F}_q^d$, $d \ge 2$. 
Then for any $m\in {\mathbb F}_q^d,$ we have
\[ \widehat{S_j}(m) = q^{-1} \delta_0(m) + q^{-d-1}\eta^d(-1) G^d \sum_{s \in {\mathbb F}_q^*} 
\eta^d(s)\chi\Big(js+ \frac{\|m\|}{4s}\Big),\]
where $G$ denotes the Gauss sum, $\eta$ is the quadratic character of $\mathbb F_q^*,$ and $\delta_0(m)=1$ if $m=(0, \ldots, 0)$ and $\delta_0(m)=0$ otherwise.
\end{lemma} 

\vskip.125in 
\subsection{A lower bound of $\displaystyle\sum_{t\in \mathbb F_q} \nu(t)\nu(rt)$ for even dimensions $d\ge 2$} Since $\displaystyle\sum_{t\in \mathbb F_q} \lambda(rt)=0$ for $r\ne 0$, it follows from Lemma \ref{sizeSt} that 
$$ I := q^{-2d}|E|^4 \sum_{t\in \mathbb F_q}|S_t||S_{rt}|
=q^{-2d}|E|^4 \left(q^{2d-1} + q^{d-2} \sum_{t\in \mathbb F_q} \lambda(t)\lambda(rt)\right)$$
$$=q^{-2d}|E|^4 \left(q^{2d-1} + q^{d-2} \lambda^2(0) + q^{d-2} \sum_{t\ne 0} \lambda(t)\lambda(rt)\right)$$ 
$$= q^{-2d}|E|^4 \left( q^{2d-1}+ q^{d-2} (q-1)^2 + q^{d-2} (q-1)\right).$$
Hence, we obtain
\begin{equation}\label{lowI}I= q^{-1} |E|^4 + q^{-d}|E|^4 - q^{-d-1}|E|^4.\end{equation}

In order to estimate the remaining terms, we need the following calculations. 
\begin{lemma} \label{sumsphere} Suppose that $m \not= \vec{0}$ in $\mathbb F_q^d, d\ge 2.$ Then for any $r\ne 0,$ we have 
\begin{equation}\label{eqze} \sum_{t\in\mathbb F_q} \widehat{S}_{rt}(m)=0\end{equation} 
and
\begin{equation}\label{eqq} \sum_{t\in \mathbb F_q} \lambda(t) \widehat{S_{rt}}(m)=q \widehat{S_0}(m),\end{equation}
where $\lambda(t)$ is defined as in Lemma \ref{sizeSt}.
\end{lemma} 

To see this, observe that the left hand side of \eqref{eqze} equals 
$$ q^{-d} \sum_{t\in \mathbb F_q} \sum_{x\in \mathbb F_q^d} \chi(-x \cdot m) S_{rt}(x)=q^{-d} \sum_{x\in \mathbb F_q^d} \chi(-x \cdot m) \sum_{t\in \mathbb F_q} S_{rt}(x)$$
$$=q^{-d} \sum_{x\in \mathbb F_q^d} \chi(-x \cdot m)=0$$ since $m \not=(0,\ldots,0)$. Hence the equation \eqref{eqze} follows. 
By the definition of $\lambda(t)$, 
$$ \sum_{t\in \mathbb F_q} \lambda(t) \widehat{S_{rt}}(m) = (q-1) \widehat{S_0}(m) - \sum_{t\ne 0} \widehat{S_{rt}}(m) = (q-1) \widehat{S_0}(m) - \sum_{t\in \mathbb F_q} \widehat{S_{rt}}(m)+\widehat{S_0}(m).$$ Then the equation \eqref{eqq} follows by \eqref{eqze}.
This completes the proof of Lemma \ref{sumsphere}. 
\vskip.125in 

We shall also need the following orthogonality lemma. 
\begin{lemma}\label{keyorthogonality} Suppose that $r\in \mathbb F_q^*$ and $m, m' \in \mathbb F_q^d.$ If $d\ge 2$ is even, then we have
$$ \sum_{t\in \mathbb F_q} \widehat{S_t}(m)~ \widehat{S_{rt}}(m')
=\left\{\begin{array}{ll} q^{-1} \delta_0(m)~\delta_0(m') 
+ q^{-d}-q^{-d-1}\quad &\mbox {if} \quad \|m\|=r\|m'\|\\                        -q^{-d-1} \quad &\mbox{if} \quad  \|m\|\ne r\|m'\|. \end{array}\right.$$
\end{lemma}
The proof shall be given at the end of the paper (see Lemma \ref{putlem}). With the lemmas in tow, we are ready to handle terms $II,III$ and $IV$. In view of Lemmas \ref{sizeSt} and \ref{sumsphere}, if $m'\ne \vec{0},$ then

$$ \sum_{t\in \mathbb F_q} |S_t| \widehat{S_{rt}}(m')=q^{d-1} \sum_{t\in \mathbb F_q} \widehat{S_{rt}}(m') + q^{\frac{d-2}{2}}\eta\left((-1)^{d/2}\right) \sum_{t\in \mathbb F_q} \lambda(t) \widehat{S_{rt}}(m')=q^{\frac{d}{2}}\eta\left((-1)^{d/2}\right)\widehat{S_0}(m').$$
Using this equation, it follows that
$$ II:=q^d|E|^2 \sum_{m'\ne \vec{0}} |\widehat{E}(m')|^2 \sum_{t\in \mathbb F_q} |S_t| \widehat{S_{rt}}(m')= q^{\frac{3d}{2}} \eta\left((-1)^{d/2}\right) |E|^2 \sum_{m'\ne \vec{0}} |\widehat{E}(m')|^2\widehat{S_0}(m').$$
By the same argument, it is not difficult to see that $II=III$. Namely, we have
\begin{equation}\label{lowII+III} II+III=2q^{\frac{3d}{2}} \eta\left((-1)^{d/2}\right) |E|^2\sum_{m\ne \vec{0}} |\widehat{E}(m)|^2\widehat{S_0}(m).\end{equation}

 We now move on to term $\displaystyle IV:=q^{4d} \sum_{m,m'\ne \vec{0}} |\widehat{E}(m)|^2|\widehat{E}(m')|^2 \sum_{t\in \mathbb F_q} \widehat{S_{t}}(m)\widehat{S_{rt}}(m')$.
Using Lemma \ref{keyorthogonality}, we can write $IV=A+B$, where 
$$ A= -q^{3d-1} \sum_{||m||\ne r||m'||, m,m' \not=\vec{0}} {|\widehat{E}(m)|}^2 {|\widehat{E}(m')|}^2$$ and 
$$ B=(q^{3d}-q^{3d-1}) \sum_{||m|| =r||m'||, m,m' \not=\vec{0}} {|\widehat{E}(m)|}^2 {|\widehat{E}(m')|}^2.$$
It follows that
$$ IV=A+B=q^{3d}\sum_{||m|| =r||m'||, m,m' \not=\vec{0}} {|\widehat{E}(m)|}^2 {|\widehat{E}(m')|}^2
-q^{3d-1} \sum_{m,m' \not=\vec{0}} {|\widehat{E}(m)|}^2 {|\widehat{E}(m')|}^2:=A'-B'.$$

Combining this with equations \eqref{lowI}, \eqref{lowII+III}, we obtain that if $d\ge 2$ is even and $r\ne 0,$ then
\begin{align*} &\sum_{t\in \mathbb F_q} \nu(t)\nu(rt)=I+II+III+IV\\
&= \left(q^{-1} |E|^4 + q^{-d}|E|^4 - q^{-d-1}|E|^4\right) +2q^{\frac{3d}{2}} \eta\left((-1)^{d/2}\right) |E|^2\left(\sum_{m\ne \vec{0}} |\widehat{E}(m)|^2\widehat{S_0}(m)\right) +(A'-B').
\end{align*}
Notice that each term above is a real number. It follows that	
\begin{align*}\sum_{t\in \mathbb F_q} \nu(t)\nu(rt)&\ge q^{-1}|E|^4 - 2 q^{\frac{3d}{2}}|E|^2 \left(\max_{m\ne \vec{0}} |\widehat{S_0}(m)|\right) \left(\sum_{m\in \mathbb F_q^d} |\widehat{E}(m)|^2\right) + (A'-B')\\
&=q^{-1}|E|^4 -2q^{\frac{d}{2}}|E|^3 \left(\max_{m\ne \vec{0}} |\widehat{S_0}(m)|\right)+ (A'-B'),\end{align*}
where we used the Plancherel theorem which states $\displaystyle\sum_{m\in \mathbb F_q^d}|\widehat{E}(m)|^2=q^{-d}|E|.$ 
By definitions of $A'$ and $B'$, we see that
\begin{align*} A'-B' &\ge q^{3d} \left(\sum_{\|m\|=0, m\ne \vec{0}} |\widehat{E}(m)|^2\right)^2 - q^{3d-1} \left(\sum_{m\in \mathbb F_q^d} |\widehat{E}(m)|^2\right)^2 \\
&=q^{3d} \left(\sum_{\|m\|=0, m\ne \vec{0}} |\widehat{E}(m)|^2\right)^2 -q^{d-1}|E|^2.\end{align*}
We also see from  Lemma \ref{Lem1} that if $d\ge 2$ is even, then
$$\max_{m\ne \vec{0}} |\widehat{S_0}(m)| \le q^{-d/2}.$$
Thus we conclude that if $d\ge 2$ is even and $r\ne 0$, then
\begin{equation}\label{Evencom} \sum_{t\in \mathbb F_q} \nu(t)\nu(rt) \ge q^{-1}|E|^4 -2|E|^3 + q^{3d} \left(\sum_{\|m\|=0, m\ne \vec{0}} |\widehat{E}(m)|^2\right)^2 -q^{d-1}|E|^2.\end{equation}

\subsection{An upper bound of $\nu^2(0)$ for even dimensions $d\ge 2$} It follows that 
$$\nu(0)=q^{2d} \sum_{m\in \mathbb F_q^d} \widehat{S_0}(m) |\widehat{E}(m)|^2.$$
By Lemma \ref{Lem1}, notice that if $d\ge 2$ is even, then
$$  \widehat{S_0}(m) = q^{-1} \delta_0(m) + q^{-d-1} G^d \sum_{s \in {\mathbb F}_q^*} 
\chi(s\|m\|).$$
Then we see that
$$\nu(0)= q^{-1}|E|^2 + q^{d-1} G^d \sum_{m\in \mathbb F_q^d} |\widehat{E}(m)|^2 \left(-1+\sum_{s\in \mathbb F_q} \chi(s\|m\|)\right).$$
By  the Plancherel theorem and the orthogonality of $\chi$, 
\begin{equation*}\label{Evanu0}\nu(0)=q^{-1}|E|^2 -q^{-1}G^d|E| + q^d G^d \sum_{\|m\|=0} |\widehat{E}(m)|^2.\end{equation*}
Since $\widehat{E}(\vec{0})=q^{-d}|E|,$  we can write 
\begin{equation}\label{Evanu0}\nu(0)= q^{-1}|E|^2 -q^{-1}G^d|E|+ q^{-d}G^d |E|^2 + q^d G^d \sum_{\|m\|=0, m\ne \vec{0}} |\widehat{E}(m)|^2.\end{equation}
We shall use the following explicit form of the Gauss sum $G.$
\begin{lemma} [\cite{LN97}, Theorem 5.15] \label{Lem4}
Let $\mathbb F_q$ be a finite field with $q=p^\ell$ for an odd prime $p$ and $\ell\in \mathbb N.$ Then the Gauss sum $G$ satisfies that
$$ G=\left\{\begin{array}{ll} (-1)^{\ell-1} q^{\frac{1}{2}} \quad &\mbox {if} \quad p\equiv 1~(\mbox{mod}~4)\\                        
(-1)^{\ell-1} i^\ell q^{\frac{1}{2}} \quad &\mbox{if} \quad p\equiv 3~(\mbox{mod}~4)  . \end{array}\right.$$
\end{lemma}

Observe from Lemma \ref{Lem4} that if the dimension $d$ is even, then $G^d=\pm q^{d/2}$ where the sign depends on $d$and $q$. 
Combining this observation with \eqref{Evanu0}, and considering the sign of each term, we see that if $d$ is even, then
\begin{equation*}\label{Unuzero} \nu(0) \le \left\{\begin{array}{ll} q^{-1}|E|^2+ q^{\frac{d-2}{2}}|E| \quad &\mbox {if} \quad G^d=-q^{\frac{d}{2}}\\                        
 \displaystyle q^{-1}|E|^2+ q^{-\frac{d}{2}}|E|^2+q^{\frac{3d}{2}}\sum\limits_{\|m\|=0, m\ne \vec{0}} |\widehat{E}(m)|^2 \quad &\mbox {if} \quad G^d=q^{\frac{d}{2}}. \end{array}\right.\end{equation*}
 Assuming that $|E|\ge q^{d/2},$ we see that
 \begin{equation*}\label{Unuzero} \nu(0) \le \left\{\begin{array}{ll} 2q^{-1}|E|^2 \quad &\mbox {if} \quad G^d=-q^{\frac{d}{2}}\\                        
 \displaystyle 2q^{-1}|E|^2+q^{\frac{3d}{2}}\sum\limits_{\|m\|=0, m\ne \vec{0}} |\widehat{E}(m)|^2 \quad &\mbox {if} \quad G^d=q^{\frac{d}{2}}. \end{array}\right.\end{equation*}
 Since $\nu(0)$ is a non-negative real number, it follows that if $|E|\ge q^{d/2},$ then
 $$ \nu^2(0)\le 4q^{-2}|E|^4 +4q^{\frac{3d-2}{2}}|E|^2 \sum\limits_{\|m\|=0, m\ne \vec{0}} |\widehat{E}(m)|^2
 + q^{3d} \left(\sum\limits_{\|m\|=0, m\ne \vec{0}} |\widehat{E}(m)|^2\right)^2 .$$
 Since $\sum\limits_{\|m\|=0, m\ne \vec{0}} |\widehat{E}(m)|^2 \le \sum\limits_{m\in \mathbb F_q^d} |\widehat{E}(m)|^2 =q^{-d}|E|,$ we conclude that if $|E|\ge q^{d/2}$, then
\begin{equation}\label{UnuzeroB} \nu^2(0)\le 4q^{-2}|E|^4 +4 q^{\frac{d-2}{2}}|E|^3 + q^{3d} \left(\sum\limits_{\|m\|=0, m\ne \vec{0}} |\widehat{E}(m)|^2\right)^2 .\end{equation}
Now we are ready to complete the proof of Theorem \ref{main}.
\subsection{Complete proof of Theorem \ref{main}} We must show that \eqref{mama} holds. 
By \eqref{Evencom} and \eqref{UnuzeroB}, it is enough to show that if $|E|\ge 9q^{d/2},$ then
$$ q^{-1}|E|^4 -2|E|^3 -q^{d-1}|E|^2 > 4 q^{-2}|E|^4+ 4 q^{\frac{d-2}{2}}|E|^3.$$
It suffices to show that
$$q^{-1}|E|^4-6 q^{\frac{d-2}{2}}|E|^3 -q^{d-1}|E|^2>4 q^{-2}|E|^4.$$
If $|E|\ge 9q^{d/2}$, then we see that
$$q^{-1}|E|^4-6 q^{\frac{d-2}{2}}|E|^3 -q^{d-1}|E|^2 \ge \frac{1}{3} q^{-1}|E|^4-q^{d-1}|E|^2,$$
so it is sufficient to show that 
$$ \frac{1}{3} q^{-1}|E|^4-q^{d-1}|E|^2 > 4 q^{-2}|E|^4.$$
Observe that if $|E|\ge 9 q^{d/2} (\ge \sqrt{12} q^{d/2})$, then  
$$\frac{1}{3} q^{-1}|E|^4-q^{d-1}|E|^2  \ge  \frac{1}{4} q^{-1}|E|^4.$$
Consequently, it suffices to show that
$$\frac{1}{4} q^{-1}|E|^4 >4 q^{-2}|E|^4.$$
which holds if $q>16.$ Therefore, when $ q\le 16$, it suffices to prove the statement of Theorem \ref{main}.
More precisely, it remains to show that if $|E|\ge 9q^{d/2}$ and $q\le 16$, then $\mathbb F_q=\frac{\Delta(E)}{\Delta(E)}.$ Since $ 9q^{d/2} > 2 q^{(d+1)/2}$ for $q\le 16,$ it will be enough to prove that if 
$|E|>2q^{(d+1)/2}$, then $ \Delta(E)=\mathbb F_q.$ In \cite{IR07}, this was proved by the first listed author and Misha Rudnev. Thus the proof of Theorem \ref{main} is complete.

\section{Proof of Theorem \ref{mainOdd}}
We proceed as in the proof of Theorem \ref{main}.
As seen in \eqref{formulasumproduct},  for $r\in \mathbb F_q^+$, we can write that
$$ \sum_{t\in \mathbb F_q} \nu(t)\nu(rt)$$
$$ =\sum_{t\in \mathbb F_q} \left(q^{-d}{|E|}^2 |S_t|+q^{2d}\sum_{m\ne \vec{0}} \widehat{S_t}(m)|\widehat{E}(m)|^2  \right) \left(q^{-d}{|E|}^2 |S_{rt}|+q^{2d}\sum_{m'\ne \vec{0}} \widehat{S_{rt}}(m')|\widehat{E}(m')|^2  \right)  $$
$$ = q^{-2d}|E|^4 \sum_{t\in \mathbb F_q}|S_t||S_{rt}|  + q^d|E|^2 \sum_{m'\ne \vec{0}} |\widehat{E}(m')|^2 \sum_{t\in \mathbb F_q} |S_t| \widehat{S_{rt}}(m') $$
$$ +q^d|E|^2 \sum_{m\ne \vec{0}} |\widehat{E}(m)|^2 \sum_{t\in \mathbb F_q} |S_{rt}| \widehat{S_{t}}(m)
+ q^{4d} \sum_{m,m'\ne \vec{0}} |\widehat{E}(m)|^2|\widehat{E}(m')|^2 \sum_{t\in \mathbb F_q} \widehat{S_{t}}(m)\widehat{S_{rt}}(m')$$
$$:= \mbox{I} + \mbox{II} + \mbox{III} +\mbox{IV}.$$
The following explicit value of $|S_t|$ is given as Theorem 6.27 in \cite{LN97}.
\begin{lemma}\label{sizeStodd} Let $S_t \subset \mathbb F_q^d$ denote the sphere with radius $t\in \mathbb F_q.$ If $d\ge 3$ is odd, then 
$$|S_t|=q^{d-1} + q^{\frac{d-1}{2}} \eta\left((-1)^{(d-1)/2} t\right),$$
where $\eta$ denotes the quadratic character of $\mathbb F_q^*$ and $\eta(0)=0.$  
\end{lemma}

We record from Lemma \ref{Lem1} that if $d\ge 3$ is odd, then
for any $m\in {\mathbb F}_q^d,$
\begin{equation}\label{oddFsphere} \widehat{S_j}(m) = q^{-1} \delta_0(m) + q^{-d-1}\eta(-1) G^d \sum_{s \in {\mathbb F}_q^*} 
\eta(s)\chi\Big(js+ \frac{\|m\|}{4s}\Big).\end{equation}

\subsection{Estimate of $\displaystyle\sum_{t\in \mathbb F_q}\nu(t)\nu(rt)$ for odd dimensions $d\ge 3$} Since $\displaystyle\sum_{t\in \mathbb F_q^*} \eta(t)=0$ (by the orthogonality of $\eta$) and $\eta(0)=0,$ it follows from Lemma \ref{sizeStodd} that
$$\mbox{I}:= q^{-2d}|E|^4 \sum_{t\in \mathbb F_q}|S_t||S_{rt}|$$
$$=q^{-2d}|E|^4 \sum_{t\in \mathbb F_q} \left(q^{d-1} + q^{\frac{d-1}{2}} \eta\left((-1)^{(d-1)/2} t\right)\right) \left(q^{d-1} + q^{\frac{d-1}{2}} \eta\left((-1)^{(d-1)/2} rt\right)\right)$$  
$$=q^{-2d}|E|^4 \left(\sum_{t\in \mathbb F_q}q^{2d-2} + \sum_{t\in \mathbb F_q} q^{d-1}\eta(r)\eta^2(t)\right)=q^{-2d}|E|^4 \left(q^{2d-1} + q^{d-1}\eta(r) (q-1)\right).$$
Since $\eta(r)=1$ (by the assumption that $r\in \mathbb F_q^+$), we have
$$ \mbox{I}\ge q^{-1}|E|^4.$$ 
In order to estimate the second term $\mbox{II},$  we begin by proving the following result.
\begin{lemma}\label{OmegaV} Let $S_j$ be the sphere in $\mathbb F_q^d$ for odd $d\ge 3.$ Then for $r\ne 0$ and $m\ne \vec{0}$, we have
$$ \Omega:=\sum_{t\in \mathbb F_q} |S_t| \widehat{S_{rt}}(m) = q^{\frac{-d-3}{2}} G^{d+1} \eta\left(r\,(-1)^{(d+1)/2}\right) \left(-1+\sum_{s\in \mathbb F_q} \chi(s\|m\|)\right).$$
\end{lemma}
To prove this lemma, recall from \eqref{eqze} of Lemma \ref{sumsphere} that $\displaystyle\sum_{t\in \mathbb F_q}\widehat{S}_{rt}(m)=0$ for $r\ne 0$ and $m\ne \vec{0}.$ By Lemma \ref{sizeStodd},
$$ \Omega=\sum_{t\in \mathbb F_q} \left( q^{d-1}+q^{\frac{d-1}{2}} \eta\left((-1)^{(d-1)/2} t\right)\right) \widehat{S_{rt}}(m)=q^{\frac{d-1}{2}}\eta\left((-1)^{(d-1)/2} \right)\sum_{t\in \mathbb F_q}\eta(t)\widehat{S_{rt}}(m). $$
By using  the value of $\widehat{S_{rt}}(m)$ in \eqref{oddFsphere}, we can write
$$\Omega=q^{\frac{-d-3}{2}} \eta\left((-1)^{(d+1)/2}\right) G^d \sum_{s\ne 0} \eta(s) \chi\left(\frac{\|m\|}{4s}\right) \left(\sum_{t\in \mathbb F_q} \eta(t) \chi(rst)\right).$$
Since $\eta(0)=0$ and $ \eta(a)=\eta(a^{-1})$ for $a\ne 0,$ a simple change of variables yields 
$$ \sum_{t\in \mathbb F_q} \eta(t) \chi(rst) = \eta(rs) G$$
and thus we have 
$$ \Omega=q^{\frac{-d-3}{2}} \eta\left((-1)^{(d+1)/2}\right) G^{d+1} \eta(r) \sum_{s\ne 0} \chi(s\|m\|),$$
which completes the proof of Lemma \ref{OmegaV}. \\

By  Lemma \ref{OmegaV} and the orthogonality of $\chi,$ we see that 
$$ \mbox{II}:=q^d|E|^2 \sum_{m'\ne \vec{0}} |\widehat{E}(m')|^2 \sum_{t\in \mathbb F_q} |S_t| \widehat{S_{rt}}(m')$$
$$ =q^{\frac{d-1}{2}} |E|^2  G^{d+1} \eta\left(r\,(-1)^{(d+1)/2}\right) \sum_{m'\ne \vec{0}, \|m'\|=0} |\widehat{E}(m')|^2 $$
$$- q^{\frac{d-3}{2}} |E|^2  G^{d+1} \eta\left(r\,(-1)^{(d+1)/2}\right) \sum_{m'\ne \vec{0}} |\widehat{E}(m')|^2 := \mbox{II}_1-\mbox{II}_2.$$
Now observe from Lemma \ref{Lem4} that $G^{d+1}\in \mathbb R$ for odd $d$ and so both $\mbox{II}_1$ and $\mbox{II}_2$ are real numbers. Furthermore, both values ​​are real numbers with the same sign. 
Hence, $\mbox{II}=\mbox{II}_1-\mbox{II}_2 \ge \min\{-|\mbox{II}_1|, -|\mbox{II}_2|\}.$
Since 
$$\min\{-|\mbox{II}_1|, -|\mbox{II}_2|\} \ge - \left|q^{\frac{d-1}{2}} |E|^2  G^{d+1} \eta\left(r\,(-1)^{(d+1)/2}\right)\right| \sum_{m'\in \mathbb F_q^d} |\widehat{E}(m')|^2,$$ which is same as
$-|E|^3,$ we obtain that
$$\mbox{II}\ge -|E|^3.$$
By the same argument, it is not hard to see that $II=III$ and we also have 
$$ \mbox{III}\ge -|E|^3.$$

In order to estimate the forth term $\mbox{IV},$ we shall need the following orthogonality lemma. 
\begin{lemma}\label{keyorthogonalityodd} Suppose that $r\in \mathbb F_q^*$ and $m, m' \in \mathbb F_q^d.$ If $d\ge 3$ is odd, then we have
$$ \sum_{t\in \mathbb F_q} \widehat{S_t}(m)~ \widehat{S_{rt}}(m')
=\left\{\begin{array}{ll} q^{-1} \delta_0(m)~\delta_0(m') 
+ (q^{-d}-q^{-d-1}) \eta(r)\quad &\mbox {if} \quad \|m\|=r\|m'\|\\                        -q^{-d-1} \eta(r)\quad &\mbox{if} \quad  \|m\|\ne r\|m'\|, \end{array}\right.$$
where $\eta$ denotes the quadratic character of $\mathbb F_q^*.$
\end{lemma}
The proof shall be given at the end of the paper (see Lemma \ref{putlem}). 
By the definition of the term $\mbox{IV}$ and Lemma \ref{keyorthogonalityodd}, it follows that
$$ \mbox{IV}:=q^{4d} \sum_{m,m'\ne \vec{0}} |\widehat{E}(m)|^2|\widehat{E}(m')|^2 \sum_{t\in \mathbb F_q} \widehat{S_{t}}(m)\widehat{S_{rt}}(m')$$
$$=-q^{3d-1}\eta(r) \sum_{m,m'\ne \vec{0}, \|m\|\ne r\|m'\|} |\widehat{E}(m)|^2 |\widehat{E}(m')|^2$$
 $$+ (q^{3d}-q^{3d-1}) \eta(r) \sum_{m,m'\ne \vec{0}, \|m\|=r\|m'\|} |\widehat{E}(m)|^2 |\widehat{E}(m')|^2.$$
Since $\eta(r)=1$ (by our assumption that $r$ is a square number in $\mathbb F_q^*$), the second term above is positive. Thus we have
$$ \mbox{IV}\ge - q^{3d-1} \sum_{m, m'\in \mathbb F_q^d} |\widehat{E}(m)|^2 |\widehat{E}(m')|^2.$$
By the Plancherel theorem,  
$$\mbox{IV}\ge -q^{d-1} |E|^2.$$
Putting this together with all other estimates, we obtain that if $d\ge 3$ is odd and $r$ is a square number, then
\begin{equation}\label{productvtodd}\sum_{t\in \mathbb F_q} \nu(t)\nu(rt):=\mbox{I}+ \mbox{II}+\mbox{III}+\mbox{IV}\ge  q^{-1}|E|^4-2|E|^3-q^{d-1} |E|^2.\end{equation} 
\vskip.125in 
\subsection{Estimate of $\nu^2(0)$ for odd dimensions $d\ge 3$}
Recall that we can write 
$$\nu(0)=q^{2d} \sum_{m\in \mathbb F_q^d} \widehat{S_0}(m) |\widehat{E}(m)|^2.$$
$$= q^{2d} \widehat{S_0}(\vec{0}) |\widehat{E}(\vec{0})|^2 + q^{2d}\sum_{m\ne \vec{0}} \widehat{S_0}(m) |\widehat{E}(m)|^2:= M +R.$$
Since $|S_0|=q^{d-1}$ for odd $d\ge 3$ (see Lemma \ref{sizeStodd}), 
$$ M= q^{-d} |S_0| |E|^2 = q^{-1}|E|^2.$$
To estimate $R$, observe that
$$ R \le q^{2d}\left(\max_{m\ne \vec{0}} |\widehat{S_0}(m)|\right) \left(\sum_{m\in \mathbb F_q^d} |\widehat{E}(m)|^2\right) =\left(\max_{m\ne \vec{0}} |\widehat{S_0}(m)|\right) q^{d}|E|.$$ 
By the equation \eqref{oddFsphere},  we see that if $d\ge 3$ is odd and $m \ne \vec{0}$, then
$$ \widehat{S_0}(m)=q^{-d-1}\eta(-1) G^d \sum_{s\ne 0} \eta(s) \chi\left(\frac{\|m\|}{4s}\right).$$
Since  $\displaystyle\left|\sum_{s\ne 0} \eta(s) \chi\left(\frac{\|m\|}{4s}\right)\right| = \sqrt{q}$ for $\|m\|\ne 0$ and $0$ otherwise, we see 
$$ \max_{m\ne \vec{0}} |\widehat{S_0}(m)| \le q^{\frac{-d-1}{2}}.$$
Hence we obtain
$$ R\le q^{\frac{d-1}{2}}|E|.$$
We have seen that $\nu(0):=M+R\le q^{-1}|E|^2 + q^{\frac{d-1}{2}}|E|$ which in turn implies 
\begin{equation}\label{V0Odd}\nu^2(0)\le q^{-2}|E|^4 +2 q^{\frac{d-3}{2}}|E|^3+ q^{d-1}|E|^2,\end{equation}
since $\nu(0)$ is a non-negative integer.
\vskip.125in 

\subsection{Complete proof of Theorem \ref{mainOdd}} Let $d\ge 3$ be odd. Suppose that $r$ is a square number in $\mathbb F_q^*.$
We must show  that if $E\subset \mathbb F_q^d $ with $|E|\ge 6q^{d/2},$
 then
$$ \sum_{t\in \mathbb F_q} \nu(t)\nu(rt) > \nu^2(0).$$
By \eqref{productvtodd} and \eqref{V0Odd}, it will be enough to show that if $|E|\ge 6q^{d/2}$, then
$$ q^{-1}|E|^4-2|E|^3-q^{d-1} |E|^2> q^{-2}|E|^4 +2 q^{\frac{d-3}{2}}|E|^3+ q^{d-1}|E|^2.$$
Note that to prove this it suffices to show that
$$ q^{-1}|E|^4-4q^{\frac{d-3}{2}}|E|^3-2q^{d-1} |E|^2>q^{-2}|E|^4.$$
If $|E|\ge 6 q^{d/2}(\ge 6 q^{(d-1)/2})$, then we see that
$$ q^{-1}|E|^4-4q^{\frac{d-3}{2}}|E|^3-2q^{d-1} |E|^2\ge \frac{1}{3} q^{-1}|E|^4 -2q^{d-1} |E|^2.$$
Hence it is sufficient to show that if $|E|\ge 6q^{d/2},$ then
$$\frac{1}{3} q^{-1}|E|^4 -2q^{d-1} |E|^2 > q^{-2}|E|^4.$$
Observe that if $|E|\ge 6q^{d/2} (\ge \sqrt{24}q^{d/2}),$ then
$$\frac{1}{3} q^{-1}|E|^4 -2q^{d-1} |E|^2\ge \frac{1}{4} q^{-1}|E|^4.$$
In conclusion, it is enough to prove that if $|E|\ge 6 q^{d/2},$ then
$$\frac{1}{4} q^{-1}|E|^4 > q^{-2}|E|^4,$$
which is clearly true provided that $q>4.$ For this reason, it suffices to prove the statement of Theorem \ref{mainOdd} in the case when $q\le 4$ and $|E|\ge 6 q^{d/2}.$ In other words, our task is to prove that
if $|E|\ge 6q^{d/2}$ for $q\le 4$, then $\mathbb F_q=\frac{\Delta(E)}{\Delta(E)}.$ Since $ 6q^{d/2}> 2 q^{(d+1)/2}$ for $q\le 4$, it will be enough to show that if $|E|> 2q^{(d+1)/2}$, then $\Delta(E)=\mathbb F_q.$ 
This is a well-known result on the Erd\H{o}s-Falconer distance problem due to the first listed author and Misha Rudnev \cite{IR07}.
Thus we finish the proof of Theorem \ref{mainOdd}. 
\section{Proof of Lemmas \ref{keyorthogonality} and \ref{keyorthogonalityodd}} 
\vskip.125in 
We begin by proving the following lemma.
\begin{lemma}\label{Thm3}
Let $r\in \mathbb F_q^*$ and $m, m'\in \mathbb F_q^d.$ Then we have
$$ \sum_{t\in \mathbb F_q} \widehat{S_t}(m)~ \widehat{S_{rt}}(m')
=\left\{\begin{array}{ll} q^{-1} \delta_0(m)~\delta_0(m') + q^{-2d}G^{2d}\eta^d(-r)(1-q^{-1})\quad &\mbox {if} \quad \|m\|=r\|m'\|\\                        -q^{-2d-1}G^{2d} \eta^d(-r)\quad &\mbox{if} \quad  \|m\|\ne r\|m'\|. \end{array}\right.$$
\end{lemma}
\begin{proof}
By Lemma \ref{Lem1}, we have
$$  \widehat{S_t}(m)=q^{-1} \delta_0(m) + q^{-d-1} \eta^d(-1)G^d \sum_{s \in \mathbb F_q^*} \eta^d(s)
\chi\Big(ts+ \frac{\|m\|}{4s}\Big) := A(t)+ B(t)$$
and
$$ \widehat{S_{rt}}(m')=q^{-1} \delta_0(m')+ q^{-d-1} \eta^d(-1)G^d \sum_{s'\in \mathbb F_q^*} \eta^d(s')\chi\Big(rts'+ \frac{\|m'\|}{4s'}\Big) :=C(t) + D(t). $$
Since $\displaystyle\sum_{t\in \mathbb F_q} A(t)D(t) =0 = \sum_{t\in \mathbb F_q} B(t)C(t)$ by the orthogonality of $\chi,$ we have
$$ \sum_{t\in \mathbb F_q} \widehat{S_t}(m)~ \widehat{S_{rt}}(m') =  \sum_{t\in \mathbb F_q} A(t)C(t) + \sum_{t\in \mathbb F_q} B(t) D(t)$$
$$= q^{-1} \delta_0(m) \delta_0(m') + q^{-2d-2} G^{2d} \sum_{s,s'\in \mathbb F_q^*} \eta^d(s)\eta^d(s')\chi\left(\frac{\|m\|}{4s} + \frac{\|m'\|}{4s'}\right) \sum_{t\in \mathbb F_q} \chi(t(s +rs'))$$
$$= q^{-1} \delta_0(m) \delta_0(m') + q^{-2d-1} G^{2d} \sum_{s\in \mathbb F_q^*}\eta^d(-s^2/ r) \chi\left(\frac{\|m\|}{4s}-\frac{r\|m'\|}{4s}\right) $$ 
$$=q^{-1} \delta_0(m) \delta_0(m') + q^{-2d-1} G^{2d} \eta^d(-r) \sum_{s\in \mathbb F_q^*} \chi(s(\|m\|-r\|m'\|))$$
$$=q^{-1} \delta_0(m) \delta_0(m') + \left[q^{-2d-1} G^{2d} \eta^d(-r) \sum_{s\in \mathbb F_q} \chi(s(\|m\|-r\|m'\|))\right] - q^{-2d-1} G^{2d} \eta^d(-r).$$
Thus the statement follows by the orthogonality of $\chi.$
\end{proof}
As a corollary of Lemma \ref{Thm3}, one can deduce Lemmas \ref{keyorthogonality} and \ref{keyorthogonalityodd} which can be restated as follows.
\begin{lemma}\label{putlem} Suppose that $r\in \mathbb F_q^*$ and $m, m' \in \mathbb F_q^d.$ If $d\ge 2$ is even, then we have
$$ \sum_{t\in \mathbb F_q} \widehat{S_t}(m)~ \widehat{S_{rt}}(m')
=\left\{\begin{array}{ll} q^{-1} \delta_0(m)~\delta_0(m') 
+ q^{-d}-q^{-d-1}\quad &\mbox {if} \quad \|m\|=r\|m'\|\\                        -q^{-d-1} \quad &\mbox{if} \quad  \|m\|\ne r\|m'\|. \end{array}\right.$$
On the other hand, if $d\ge 3$ is odd, then we have
$$ \sum_{t\in \mathbb F_q} \widehat{S_t}(m)~ \widehat{S_{rt}}(m')
=\left\{\begin{array}{ll} q^{-1} \delta_0(m)~\delta_0(m') 
+ (q^{-d}-q^{-d-1}) \eta(r)\quad &\mbox {if} \quad \|m\|=r\|m'\|\\                        -q^{-d-1} \eta(r)\quad &\mbox{if} \quad  \|m\|\ne r\|m'\|. \end{array}\right.$$
\end{lemma}
\begin{proof} Suppose that $d\ge 2$ is even. Then $\eta^d=1.$ By Lemma \ref{Lem4}, we see that $G^{2d}=q^d$ for even $d\ge 2.$ Thus the statement follows by Theorem \ref{Thm3}.\\
Next, assume that $d\ge 3$ is odd. Then $\eta^d=\eta.$ Hence, by Theorem \ref{Thm3} it suffices to show that $G^{2d}\eta(-1)=q^d$ for odd $d\ge 3.$ This equality follows by combining Lemma \ref{Lem4} with the facts that $\eta(-1)=1$ for $q\equiv 1~(\mbox{mod}~4),$ and $\eta(-1)=-1$ for $q\equiv 3~(\mbox{mod}~4).$
\end{proof}

\end{document}